\documentclass[reqno]{amsart}

\usepackage[hidelinks]{hyperref}
\usepackage{amsthm,amsmath,amssymb}
\usepackage{mathtools}
\usepackage{cite}
\usepackage{microtype}
\usepackage{accents}
\usepackage{tikz}

\newcommand{\tbar}[1]{\accentset{\rule{.35em}{.6pt}}{#1}}
\newcommand{\ubar}[1]{\underaccent{\rule{.35em}{.6pt}}{#1}}
\newcommand{\Hom}{\mathsf{Hom}}
\newcommand{\Obj}{\mathsf{Obj}}
\newcommand{\sal}{\mathrm{Sal}}
\newcommand{\hsal}{\widehat{\mathrm{Sal}}}
\newcommand{\rk}{\mathrm{rk}}

\title[Simplicial pseudohyperplane arrangements]{Simplicial pseudohyperplane arrangements\\ give weak Garside groups}
\author{Katherine M. Goldman}

\newtheorem{thm}{Theorem}
\newtheorem{prop}[thm]{Proposition}
\newtheorem{lemma}[thm]{Lemma}
\newtheorem{cor}[thm]{Corollary}

\newtheorem*{thm*}{Theorem}
\newtheorem*{prop*}{Proposition}

\theoremstyle{definition}
\newtheorem{defn}[thm]{Definition}

\let\c\mathcal

\begin{document}

\begin{abstract} 
    In this note we connect the language of Bessis's Garisde categories with Salvetti's metrical-hemisphere complexes in order to find new examples of weak Garside groups. As our main example, we show that the fundamental group of the (appropriately defined) complexified complement of a pseudohyperplane arrangement is a weak Garside group. As a consequence of the Folkman-Lawrence topological realization theorem, we also show that the fundamental group of the Salvetti complex of a (``simplicial'') oriented matroid is a weak Garside group. This provides novel examples of weak Garside groups.
\end{abstract}

\maketitle

Weak Garside groups were formalized by Bessis in \cite{bessis2006garside}. A traditional Garside group is a group endowed with a certain combinatorial structure which provides fruitful information about the group. 
The structure on a Garside group generalizes readily to groupoids, giving Bessis's Garside groupoids. A weak Garside group is a group which is equivalent (as a category) to a Garside groupoid\footnote{We clarify that a ``weak Garside group'' is not the same as a ``quasi Garside group''. A quasi Garside group may be thought of as an ``infinite-type'' Garside group.}. In general, weak Garside groups may not be Garside groups, but the weak Garside property has shown to be rather fruitful on its own. Many results for Garside groups can be easily extended to weak Garside groups, and some recent work has been done to study weak Garside groups in their own right.
Some notable properties for a group $G$ which is a finite-type (weak) Garside group include
\begin{enumerate}
    \item The word problem for $G$ is solvable \cite{dehornoy1999gaussian}.
    \item $G$ is biautomatic \cite{dehornoy1999gaussian}; in particular, the $k$-th order homological Dehn function and homotopical Dehn function of $G$ are Euclidean \cite{wenger2005isoperimetric,behrstock2019combinatorial}. 
    \item $G$ has a finite $K(G,1)$ \cite{bessis2006garside}; in particular, $G$ is torsion-free.
    \item $G$ is Helly \cite{huang2021helly}; in particular, this implies
    \begin{itemize}
        \item $G$ admits a geometric action on an injective metric space,
        \item $G$ admits an EZ-boundary and a Tits boundary,
        \item The Farrell-Jones conjecture with finite wreath products holds for $G$, and
        \item The coarse Baum-Connes conjecture holds for $G$.
    \end{itemize}
\end{enumerate}

Primary examples of (weak) Garside groups come from complements of \emph{complex hyperplane arrangements}.
A complex hyperplane arrangement is a finite collection  $\c H$ of linear hyperplanes (i.e., codimension-1 linear subspaces) in $\mathbb{C}^n$. We denote the complement of this collection by \[M(\c H) = \mathbb{C}^n \setminus \bigcup_{H \in \c H} H.\]

There are two beautiful results concerning when $\pi_1(M(\c H))$ is a weak Garside group. 
The first, by Deligne \cite{deligne1972immeubles}, considers the following class: suppose $\c H$ is a finite collection of \emph{real} linear hyperplanes (i.e., codimension-1 linear subspaces of $\mathbb{R}^n$). 
We call a connected component of $\mathbb{R}^n \setminus \bigcup_{H \in \c H} H$ a \emph{chamber}, and we call $\c H$ \emph{simplicial} if each chamber is a simplicial cone (i.e., the $\mathbb{R}^+$-span of a simplex). 
Let $\c H_{\mathbb{C}} = \{\, H_{\mathbb{C}} : H \in \c H\,\}$, where $H_{\mathbb{C}}$ denotes the complexification of the linear space $H$. 
Then if $\c H$ is simplicial, $\pi_1(M(\c H_{\mathbb{C}}))$ is a weak Garside group \cite{deligne1972immeubles}. 
The prototypical example of a Garside group arising in this way is a braid group on $n$-strands \cite{garside69}, or, more generally, the spherical-type Artin groups \cite{brieskorn1972artin} (of which the braid groups are examples). 

The second major example comes from Bessis, who in \cite{bessis2015finite} proved the following. Suppose $\c H$ is the set of reflection hyperplanes for a well-generated complex reflection group. (A well-generated complex reflection group is a finite subgroup of $\mathrm{GL}_n(\mathbb{C})$ generated by exactly $n$ complex reflections; complex reflections are finite-order linear transformations of $\mathbb{C}^n$ fixing a hyperplane.) Then $\pi_1(M(\c H))$ is a weak Garside group \cite{bessis2015finite}.

One may ask if the assumption that the hyperplanes are linear in the previous examples may be dropped. That is, are there similar results for finite collections of submanifolds which are simply complex codimension-1 and intersect at the origin? While it is certainly unfeasable to expect this to be true of any arbitrary submanifolds, if we restrict ourselves to a certain class of submanifolds, called pseudohyperplanes, we show that the answer is yes; i.e., the fundamental group of the complement of a simplicial pseudohyperplane arrangement is a weak Garside group. We give the precise statements now.

These definitions are based on those found in \cite{mandel1982topology}. 
Consider an embedding $\mathbb{S}^{d-1} \hookrightarrow \mathbb{S}^{d}$ with image $S$. 
The subspace $S \subseteq \mathbb{S}^{d}$ is called  a \emph{pseudosphere} if the components of $\mathbb{S}^{d} \setminus S$ are homeomorphic to the unit ball in $\mathbb{R}^d$. 
In this case, the components of $\mathbb{S}^{d} \setminus S$ are called the (open) halfspaces of $S$.
(One can verify that there are always two such components.)
Let $\c A= \{ S_i \}_{i \in I}$ be a finite collection of pseudospheres of $\mathbb{S}^{d}$. 
For $J \subseteq I$, let $S_J = \bigcap_{j \in J} S_j$. 
We call $\c A$ an \emph{arrangement of pseudospheres (in $\mathbb{S}^{d}$)} if for each $J \subseteq I$ and $k \in I$ satisfying $S_J \not\subset S_k$, we have that $S_J \cap S_k$ is a pseudosphere in $S_J$ with halfspaces coming from the intersection of $S_J$ with the halfspaces of $S_k$. (Compare to the intersection of a hyperplane arrangement with the unit sphere.) If $\c A$ is a pseudosphere arrangement, then the cone
\begin{align*}
    c \c A = \{\,c S : S \in \c A\,\}
\end{align*}
is called a \emph{pseudohyperplane arrangement}. (The cone $cS$ on a pseudosphere $S$ is the set $cS = \{\,\alpha x : \alpha \in [0,\infty), x \in S\,\}$.) To a pseudohyperplane arrangement, one can associate a natural ``complexified complement'' (even though these are not linear subspaces) by \[
M(c \c A) = \mathbb{C}^n \setminus \bigcup_{H \in c \c A} (H + iH),\] where $H + iH = \{\,h_1 + ih_2 : h_i \in H\,\}$. One of our main results (part of Theorem \ref{thm:mainconsequence1}) is as follows (comparable to \cite{deligne1972immeubles}). 

\begin{thm*} 
    If $c \c A$ is a centrally symmetric simplicial pseudohyperplane arrangement, then $\pi_1(M(c \c A))$ is a weak Garside group.
\end{thm*}

In particular, this positively answers a question of Deshpande regarding the bi-automaticity of such groups \cite[p.~419]{deshpande2016arrangements}, as well as showing that they possesses the other properties of weak Garside groups given above.

We note that there are many pseudohyperplane arrangements which are not realizable as hyperplane arrangements. A famous example is the non-Pappus arrangement seen in Figure \ref{fig:nonpap}.

\begin{figure}
    \includegraphics[scale=0.15]{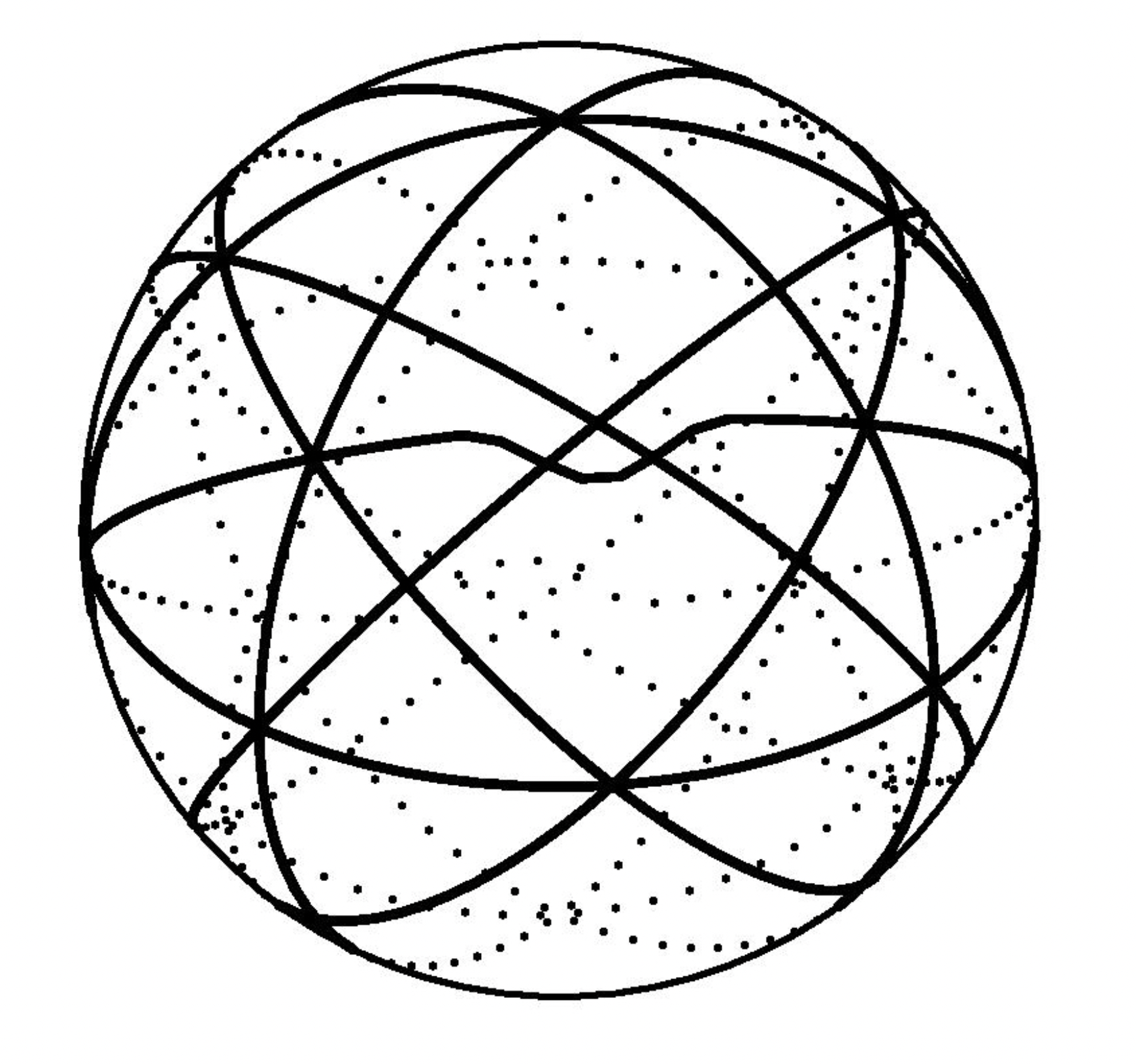}
    \caption{The non-Pappus arrangement, from \cite[Fig.~1]{deshpande2016arrangements}}
    \label{fig:nonpap}
\end{figure}

Other examples come from pseudoline arrangements in $\mathbb{RP}^2$ \cite[Ch.~2]{grunbaum1967convex}; these are finite collections of simple closed curves in $\mathbb{RP}^2$ such that every two curves intersect (transversely) in exactly one point. By lifting to $\mathbb{S}^2$ and taking cones, we have a non-realizable arrangement of pseudoplanes in $\mathbb{R}^3$. There are (at least) seven infinite families of simplicial pseudoline arrangements \cite[Ch.~3]{grunbaum1967convex}.

We also note that as a consequence of \cite{delucchi2017equivariant}, there exist pseudohyperplane arrangement complements whose fundamental groups are not isomorphic to the fundamental group of any hyperplane arrangement complement. Although, this is explicitly demonstrated using the above non-Pappus arrangement, which is not simplicial.
We believe that there exist such simplicial examples, but have not yet been able to find any in the literature.

The famous Folkman-Lawrence topological realization theorem gives a 1-1 correspondence between pseudohyperplane arrangements\footnote{Originally shown for ``pseudohemisphere arrangements'' in \cite{folkman1978oriented}, stated and proven in terms of pseudosphere/pseudohyperplane arrangements in \cite{mandel1982topology}} and oriented matroids (see Proposition \ref{prop:toprep}).
Oriented matroids are common combinatorial objects which have broad uses in mathematics, and in particular capture the combinatorics of pseudohyperplane arrangements. For a given oriented matroid $\c M$, there is an associated simplicial complex, called the \emph{Salvetti complex} $\sal(\c M)$ \cite[\S 5.5]{ziegler1987algebraic}. This complex turns out to be homotopy equivalent to $M(c \c A)$ for $c \c A$ the pseudohyperplane arrangement associated to $\c M$ \cite[Thm.~2.30]{deshpande2016arrangements}.
In Section \ref{sec:matroids}, we utilize this deep connection of pseudohyperplane arrangements with oriented matroids to show the following (Theorem \ref{thm:simpmatgar}).

\begin{thm*}
    If $\c M$ is a ``simplicial'' oriented matroid, then $\pi_1(\sal(\c M))$ is a weak Garside group.
\end{thm*}

Our proof generalizes a result of Cordovil \cite[Thm.~4.1]{cordovil1994homotopy}, who shows that under the same assumptions on an oriented matroid $\c M$, its Salvetti complex $\sal(\c M)$ is a $K(\pi,1)$. Although the result in said article does not show further group theoretic or geometric properties of $\pi_1(\sal(\c M))$, and in particular does not discuss Garside structures or normal forms as in \cite{deligne1972immeubles}. In particular, the result shown in this paper that $\pi_1(\sal(\c M))$ is a weak Garside group can be thought of as a completion of the generalization of Deligne's work to the oriented matroids.

Our main technical tool, defined in Section \ref{subsec:MHCs}, are the 
metrical-hemisphere complexes introduced by Salvetti in \cite{salvetti1993homotopy}. These complexes act as a basis for generalizing the construction of the ``Salvetti complex'' of a hyperplane arrangement to the pseudosphere/pseudohyperplane arrangements. 
It is shown in \cite{deshpande2016arrangements} that, analagous to the hyperplane case, the Salvetti complex of a pseudohyperplane arrangement is homotopic to its complexified complement. 
We may thus adapt many of the arguments involving the Salvetti complex of a hyperplane arrangement to those of pseudohyperplane arrangements---and in fact, our arguments may be applied to the Salvetti complex of any (suitable) metrical-hemisphere complex. This is the content of Section \ref{sec:simpMH}. See Theorem \ref{thm:fismhgarside} and Corollary \ref{thm:salcpxgarside} for the full statement of our most general results.

The paper is structured as follows. In Section 1, we provide the relevant background information regarding Garside categories and metrical hemisphere complexes. In Section 2, we show that the Salvetti complex of appropriate metrical hemisphere complexes gives rise to a natural categorical Garside structure, and the isotropy group of this Garside structure is isomorphic to the fundamental group of the Salvetti complex. In Section 3, we apply this result to simplicial pseudohyperplane arrangements and (simplicial) oriented matroids.

\section{Background}

\subsection{Garside categories}

The material in this section originally comes from \cite{bessis2006garside}, but  our statements are primarily adapted from \cite{haettel2022lattices}.
We recall some notions from basic category theory first, mainly to establish notation. Throughout this paper, we assume all categories are small.

Let $\c C$ be a (small) category. The set of objects of $\c C$ is denoted $\Obj(\c C)$ and the set of morphisms is denoted $\Hom(\c C)$. We denote the identity functor of $\c C$ by $1_{\c C}$. 
For $x,y \in \Obj(\c C)$, let $\c C_{x\rightarrow}$ denote the set of morphisms starting at $x$, $\c C_{\rightarrow y}$ the set of morphisms ending at $y$, and $\c C_{x\rightarrow y}$ the set of morphisms from $x$ to $y$.
If $f \in \c C_{x \rightarrow y}$, we sometimes write $x \xrightarrow{f} y$. If $x  \xrightarrow{f} y  \xrightarrow{g} z$ then the composition of $f$ and $g$ is $x \xrightarrow{fg} z$ (so we adopt the ``path'' style of the order of composition rather than the ``function'' style).

We recall that a natural transformation $\tau$ between functors $F, G: \c C \to \c D$, which we denote by $\tau : F \to G$, is a mapping from $\Obj(\c C)$ to $\Hom(\c D)$ which for $x,y \in \Obj(\c C)$ and $x \xrightarrow{f} y$ satisfies 
\begin{enumerate}
    \item $F(x) \xrightarrow{\tau(x)} G(x)$, and
    \item $ F(f)\tau(y) = \tau(x) G(f)$.
\end{enumerate}

For $f,g \in \Hom(\c C)$, we write $f \preccurlyeq g$ if there is an $h \in \Hom(\c C)$ such that $g = fh$.
We write $g \succcurlyeq f$ if there is an $h \in \Hom(\c C)$ such that $g = hf$. A nontrivial $f \in \Hom(\c C)$ which cannot be factored into two nontrivial factors is called an atom.

The category $\c C$ is cancellative if for $a,b,f,g \in \Hom(\c C)$, $afb = agb$ implies $f = g$. It is homogeneous if there is a function $\ell : \Hom(\c C) \to \mathbb{Z}_{\geq 0}$ such that $\ell(fg) = \ell(f) + \ell(g)$, and $\ell(f) = 0$ if and only if $f$ is trivial. 
Note that if $\c C$ is homogeneous, then $(\c C_{x \rightarrow}, \preccurlyeq)$ and $(\c C_{\rightarrow y}, \succcurlyeq)$ are posets.

\begin{defn}
A \emph{Garside category} is a homogeneous cancellative category $\c C$ equipped with an automorphism $\phi : \c C \to \c C$ and a natural transformation $\Delta : 1_{\c C} \to \phi$ satisfying the following properties.
\begin{enumerate}
     \item For all $x,y \in \Obj(\c C)$, the posets $(\c C_{x \rightarrow}, \preccurlyeq)$ and $(\c C_{\rightarrow y}, \succcurlyeq)$ are lattices.
     \item All atoms in $\Hom(\c C)$ are \emph{simple} ($x \xrightarrow{f} y$ is simple if there exists an $y \xrightarrow{f^*} \phi(x)$ so that $ff^* = \Delta(x)$).
 \end{enumerate} 
 Sometimes for $x \in \Obj(\c C)$, we write $\Delta_x = \Delta(x)$ and $\Delta^x = \Delta(\phi^{-1}(x))$. If $x$ is clear from context, we may write $\Delta = \Delta_x$. If the set of simple morphisms is finite, we call $\c C$ \emph{finite type}. 

Associated to a Garside category $\c C$ is a groupoid $\c G$ into which $\c C$ embeds, obtained by adding formal inverses to all morphisms in $\c C$. We call such a $\c G$ a \emph{Garside groupoid}. For $x \in \Obj(\c G)$, we sometimes write $\c G_x$ for $\c G_{x \rightarrow x}$, and call $\c G_x$ the \emph{isotropy group at $x$}.
A \emph{weak Garside group} is a group which is isomorphic to the isotropy group of some object of a Garside groupoid. A \emph{Garside group} is a Garside groupoid with one object. A \emph{Garside monoid} is a Garside category with one object. Each of these objects will be called \emph{finite type} if the corresponding Garside category is finite type. 
\end{defn}

\subsection{Metrical Hemisphere Complexes}
\label{subsec:MHCs}

The following comes from \cite{salvetti1993homotopy}.

Let $Q$ be a connected cell complex with poset of (closed) cells $\c Q$. 
For a cell $e \in \c Q$, let $\c Q(e)$ denote the set of cells of $\c Q$ contained in $e$ and let $V(e)$ denote the vertices of $e$. The set of $i$-cells of $Q$ is denoted $\c Q^{(i)}$ and the $i$-skeleton of $Q$ is denoted $Q^{(i)}$. 
% Note that $Q^{(i)}$ is the union of the cells of $\c Q^{(i)}$. 
For vertices $v,w$ of $Q$, we let $d(v,w)$ denote the distance between $v$ and $w$ in $Q^{(1)}$ with each edge given length 1 (so $d(v,w)$ is the combinatorial length of the shortest edge path from $v$ to $w$).
 
\begin{defn} \label{def:qmh}
    A regular cell complex $Q$ is quasi-metrical-hemisphere (QMH) if there exist maps $\ubar w, \tbar w : \c Q^{(0)} \times \c Q \to \c Q^{(0)}$ such that 
    \begin{enumerate}
        \item[(i)] For each cell $e$, the functions $\ubar w(\cdot , e)$ and $\tbar w(\cdot, e)$ are maps from $\c Q^{(0)}$ to $V(e)$,
        \item[(ii)] The vertices $\ubar w(v , e)$ and $\tbar w(v, e)$ are the unique vertices of $e$ satisfying
        \begin{align*}
            d(v, \ubar w(v , e)) = \min_{w \in V(e)} d(v,w) \quad \text{ and } \quad d(v, \tbar w(v , e)) = \max_{w \in V(e)} d(v,w),
        \end{align*}
        and
        \item[(iii)] For all vertices $v$ and cells $e$ of $Q$, and all cells $e' \in \c Q(e)$, 
        \begin{align*}
            \ubar w(v, e') &= \ubar w(\ubar w(v,e), e') =\tbar w(\tbar w(v,e), e'), \quad \text{ and} \tag{A} \\
            \tbar w(v,e') &= \ubar w(\tbar w(v,e),e') = \tbar w(\ubar w(v,e),e'). \tag{B}
        \end{align*}
    \end{enumerate}
\end{defn}

\begin{prop}
    If $Q$ is a regular cell complex satisfying (i) and (ii) above, then (iii) (i.e., (A) and (B) together) is equivalent to
    \begin{align*}
        d(v, \tbar w(v,e)) &= d(v,w) + d(w, \tbar w(v,e)) %\tag{C}
    \end{align*}
    for all $v \in Q^{(0)}$, $e \in \c Q$, and $w \in V(e)$. Moreover, if $Q$ is QMH, then
    \begin{align*}
        d(v,w) &= d(v, \ubar w(v,e)) + d(\ubar w(v,e), w) %\tag{D}
    \end{align*}
    for all $v \in Q^{(0)}$, $e \in \c Q$, and $w \in V(e)$.
\end{prop}

\begin{defn} \label{def:lmh}
    A regular cell complex $Q$ is local-metrical-hemisphere (LMH) if every cell $e \in \c Q$ (viewed as a subspace of $Q$) is QMH, with respective functions denoted $\tbar w_e$ and $\ubar w_e$, such that the following compatibility condition holds: If $f \in \c Q(e) \cap \c Q(e')$ and $v \in V(e) \cap V(e')$, then
    \begin{align*}
        \ubar w_{e}(v,f) &= \ubar w_{e'}(v,f) \quad \text{ and } \quad 
        \tbar w_{e}(v,f) = \tbar w_{e'}(v,f) %\tag{E}
    \end{align*}
\end{defn}

There are QMH complexes which are not LMH complexes. Take for example the complex in Figure \ref{fig:QMHnotLMH} consisting of an octagonal $2$-cell $e$ glued along a square $2$-cell $e'$ as indicated. One may readily verify that this complex is QMH. However, $\ubar w_e(v,f) = v_1 \not= v_2 = \ubar w_{e'}(v,f)$, so the complex is not LMH.

\begin{figure}[t!]
\begin{minipage}{.5\textwidth}
\centering
\includegraphics{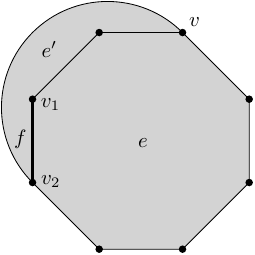}
\caption{A complex which is QMH but not LMH \cite[Fig.~2a]{salvetti1993homotopy}}
\label{fig:QMHnotLMH}
\end{minipage}%
\begin{minipage}{.5\textwidth}
\centering
\includegraphics{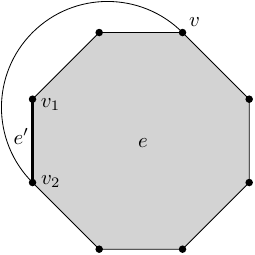}
\caption{A complex which is QMH and LMH but not MH \cite[Fig.~2b]{salvetti1993homotopy}}
\label{fig:QMHandLMHnotMH}
\end{minipage}
\end{figure}

\begin{defn}
    A regular cell complex $Q$ is metrical-hemisphere (MH) if it is both QMH and LMH, and for all $e \in \c Q$, $e' \in \c Q(e)$, and $v \in V(e)$, we have
    \begin{align*}
        \ubar w(v, e') &= \ubar w_e(v, e'), \quad \text{ and } \quad \tbar w(v,e') = \tbar w_e(v, e'). %\tag{F}
    \end{align*}
\end{defn}

Important examples of MH complexes come from hyperplane arrangements. This is explained in more detail in Section \ref{sec:matroids}.
There are complexes which are both QMH and LMH but not MH. For example, see Figure \ref{fig:QMHandLMHnotMH}, consisting of an octagonal 2-cell $e$ and a 1-cell glued as indicated. One may easily check that this complex is QMH and LMH. However, $\ubar w_e(v,e') = v_1 \not= v_2 = \ubar w(v,e')$, so it is not MH.

To a QMH or LMH complex $Q$ there is a regular cell complex $X = \sal(Q)$, which we call the \emph{Salvetti complex} of $Q$ built in the following way.

Set $X^{(0)} = Q^{(0)}$. For each edge $e$ of $Q$ with vertices $v_1,v_2$, attach two 1-cells $\langle e; v_1 \rangle$, $\langle e; v_2 \rangle$ to $v_1,v_2$ in $X$ (so that $\partial \langle e; v_i \rangle = \{v_1,v_2\}$). We choose an orientation on $\langle e, v_i \rangle$ so that $v_i$ is the source of this edge. This fully describes the 1-skeleton of $X$ as an oriented graph. 

The $i$-skeleton is constructed by induction. Suppose we have already constructed the $j$-skeleton of $X$ for $1 \leq j \leq i-1$. A cell of $X^{(j)}$ will be denoted by $\langle e; v \rangle$ for $e \in \c Q^{(j)}$ and $v \in \c Q(e)$. Then for each $i$-cell $e$ of $Q$ and $v \in V(e)$, let $\phi(v,e) : \c Q(e) \setminus \{e\} \to X^{(i-1)}$ be the map taking a cell $e'$ of $\partial e$ to $\langle e', \ubar w(v, e') \rangle$.
By the definition of a QMH complex (specifically, (A) and (B)), we have that
\begin{align*}
    \partial  e \cong \bigcup_{e' \in \c Q(e) \setminus \{e\}} \langle e'; \ubar w(v,e') \rangle.
\end{align*}
So for $i$-cells $e$ of $Q$ and vertices $v$ of $e$, we may attach an $i$-cell, denoted $\langle e; v \rangle$, via $\phi(e,v)$. 

Some properties of $\sal(Q)$ are summarized in the following

\begin{prop} \label{prop:summaryofX}
    Let $Q$ be QMH (respectively LMH, MH). Then $\sal(Q)$ is QMH (LMH, MH, resp.). There is a natural map $\psi$ from cells of $\sal(Q)$ to cells of $Q$ given by $\psi(\langle e;v \rangle) = e$. We also have
    \begin{align*}
        \partial \langle e;v \rangle = \bigcup_{e' \in \c Q(e) \setminus \{e\}} \langle e'; \ubar w(v,e') \rangle.
    \end{align*}
    The restriction of $\psi$ to $\sal(Q)^{(0)}$ is an isomorphism onto $Q^{(0)}$. In general, for an $i$-cell $e$ of $Q$, the preimage $\psi^{-1}(e)$ consists of $|V(e)|$ $i$-cells of $\sal(Q)$, i.e., there is one distinct $i$-cell for every vertex of $e$.
\end{prop}

\subsection{A categorical structure for MH complexes}

In this section, let $Q$ be a QMH, LMH, or MH complex with maps $\tbar w$ and $\ubar w$ (or $\tbar w_e$ and $\ubar w_e$) as in Definition \ref{def:qmh} (resp.~Definition \ref{def:lmh}), and let $X = \sal(Q)$ be the Salvetti complex for $Q$. In this section we define a category $\c C = \c C(Q)$ based on $Q$ (or more specifically, on $X$). 

\begin{defn}
    A \emph{positive (resp.~negative) path} in $X$ is an edge path $\gamma = (e_1,\dots,e_k)$ 
    which is a positive (resp.~negative) path in the 1-skeleton of $X$ when $X^{(1)}$ is viewed as an oriented graph with orientation described above. 
    A positive path $\gamma$ has an associated negative path $\gamma^{-1}$ formed by traversing $\gamma$ in the reverse direction (so $\gamma^{-1} = (e_k^{-1}, \dots, e_1^{-1})$ where $e^{-1}$ denotes the edge $e$ with its opposite orientation).
    The length of $\gamma$ is defined to be $\ell(\gamma) = k$. A \emph{minimal positive path} in $X$ is a positive path $\gamma$ between $v_1$ and $v_2$ so that
    $\ell(\gamma) = d(v_1,v_2)$.
\end{defn}

Note that there is at least one minimal positive path between any two vertices.

The following definition will be of use later.

\begin{defn}
    Let $\gamma$ be a positive path in $X$, say $\gamma = (\langle v_1,e_1 \rangle, \dots, \langle v_k, e_k \rangle)$, where each $e_i$ is an edge of $Q$ from $v_i$ to $v_{i+1}$, then the path $\gamma^\circ$ \emph{opposite to} $\gamma$ is given by $\gamma^\circ(\langle v_{k+1}, e_k \rangle, \langle v_k, e_{k-1}\rangle, \dots, \langle v_2,e_1 \rangle)$. 
\end{defn}

\begin{defn}
    Let $\gamma_1, \gamma_2$ be positive paths in $X$ with the same initial and terminal vertices. We say $\gamma_1$ and $\gamma_2$ are \emph{elementarily equivalent} if we can write 
    \begin{align*}
        \gamma_1 &= \alpha \hat\gamma_1 \beta, \text{ and} \\
        \gamma_2 &= \alpha \hat\gamma_2 \beta,
    \end{align*}
    where
    \begin{enumerate}
        \item[(i)] $\alpha$ and $\beta$ are positive paths, and
        \item[(ii)] $\hat\gamma_1$ and $\hat \gamma_2$ are minimal positive paths from a vertex $v$ of $X$ to the vertex $\tbar w(v,e)$, where $e$ is a cell of $Q$ containing $v$.
    \end{enumerate}
    That is, we can derive $\gamma_2$ from $\gamma_1$ by replacing a minimal positive subpath (from $v$ to $\tbar w(v,e)$) with another minimal positive path. We say $\gamma_1$ and $\gamma_2$ are \emph{equivalent}, and write $\gamma_1 \sim \gamma_2$, if there is a finite sequence of elementary equivalences taking $\gamma_1$ to $\gamma_2$. The $\sim$ equivalence class of a positive path $\gamma$ is denoted $[\gamma]$.
\end{defn}

We note that elementary equivalence clearly does not change lengths, so if $\gamma_1 \sim \gamma_2$, we know that $\ell(\gamma_1) = \ell(\gamma_2)$. In other words, the length of an equivalence class of paths is well defined, which by abuse of notation we still call $\ell$. 

We can now define a category for an MH complex $Q$.

\begin{defn}
    Let $Q$ be a QMH, LMH, or MH complex. We define a category $\c C(Q)$ as follows: the objects of $\c C(Q)$ are the vertices of $Q$, and the set of morphisms from $x$ to $y$ is the set of equivalence classes of positive edge paths from $x$ to $y$. Composition is given by concatenation of paths.
\end{defn}

An immediate consequence of the definition is the following
\begin{prop} \label{prop:posetsiso}
    For any $x \in \Obj(\c C(Q))$, the map $\gamma \mapsto \gamma^\circ$ is an (order reversing) isomorphism of $(\c C_{x \rightarrow}, \preccurlyeq)$ with $(\c C_{\rightarrow x}, \succcurlyeq)$.
\end{prop}
We will study further properties of $\c C(Q)$ in the following section. 

\section{Simplicial MH complexes} \label{sec:simpMH}

We now turn our attention to the main objects of study, which generalize simplicial hyperplane arrangements.

\begin{defn}
    Let $Q$ be an MH complex. We call $Q$ \emph{simplicial} if 
    $Q$ is dual to a simplicial triangulation of a manifold.
\end{defn}

Our following terminology departs from \cite{salvetti1993homotopy} (his ``MH$^*$-complex'' is our ``flat involutive MH complex'').

\begin{defn}
    Let $Q$ be a simply connected MH (or LMH) complex. We call $Q$ \emph{flat} if every pair $\gamma_1$, $\gamma_2$ of minimal positive paths in $\sal(Q)$ with the same initial and terminal vertices are equivalent under $\sim$. If $Q$ is flat and $x,y$ are vertices of $Q$, then the (unique) equivalence class of positive minimal paths from $x$ to $y$ will be denoted $u(x,y)$.

    We say that $Q$ is \emph{involutive} if there is an automorphism $\phi : Q^{(1)} \to Q^{(1)}$ of the 1-skeleton of $Q$ of order 2 (i.e., $\phi^2 = \mathrm{id}$) such that
    \begin{enumerate}
        \item $\phi(v)$ is the unique vertex such that $d(v,\phi(v)) = \max_{w \in Q^{(0)}} d(v,w)$, and
        \item For all vertices $v,w$ of $Q$, $d(v, \phi(v)) = d(v,w) + d(w,\phi(v))$.
    \end{enumerate}
    For any vertex $x$, we let $\Delta(x)$ denote a (choice of) minimal positive path from $x$ to $\phi(x)$.
\end{defn}

If $Q$ is involutive, the involution $\phi$ induces a map on the set of positive paths. By Proposition 24 of \cite{salvetti1993homotopy}, $\phi$ preserves $\sim$, and thus induces an automorphism of $\c C(Q)$, which by abuse of notation we still call $\phi$. Again by abuse of notation, the morphism corresponding to $\Delta(x)$ in $\c C$ is denoted $\Delta(x)$. 
We note that if $Q$ is also flat, then \emph{every} minimal positive path from $x$ to $\phi(x)$ in $Q$ is equivalent to $\Delta(x)$, and thus the morphism $\Delta(x)$ ($= u(x, \phi(x))$) is determined independent from our choice of path between $x$ and $\phi(x)$ in $Q$.

We now have the terminology to state our first main theorem.

\begin{thm} \label{thm:fismhgarside}
    Suppose $Q$ is a flat, involutive, simplicial MH complex (or FISMH). Then $\c C(Q)$ is a Garside category under $\Delta$ (i.e., the map $x \mapsto \Delta(x)$) and $\phi$ as defined above.
\end{thm}

We will prove this theorem in a series of lemmas. The first follows immediately from the definitions.

\begin{lemma} \label{lem:atomsaresimple}
    Suppose $Q$ is an involutive MH complex, and let $x,y$ be vertices of $Q$. Let $f$ be a minimal positive path from $x$ to $y$, and let $f^*$ be a minimal positive path from $y$ to $\phi(x)$. Then the concatenation $ff^*$ is a minimal positive path from $x$ to $\phi(x)$. Thus, if $Q$ is also flat, $ff^* \sim \Delta(x)$.
\end{lemma}

This implies all atoms are simple, since the atoms are paths of length $1$ and hence necessarily minimal, satisfying (2) of the definition of a Garside category. 
The next lemma follows immediately from our previous discussions, namely that if $\gamma_1 \sim \gamma_2$, then $\ell(\gamma_1) = \ell(\gamma_2)$.

\begin{lemma} \label{lem:mhhom}
    If $Q$ is an MH complex, then $\c C(Q)$ is homogeneous with length function $\ell$.
\end{lemma}

The following paraphrases \cite[Thm.~31]{salvetti1993homotopy}. 

\begin{lemma} \label{lem:qcancel}
    If $Q$ is FISMH, then $\c C(Q)$ is cancellative.
\end{lemma}

We now confirm that $\Delta$ is indeed a natural transformation.

\begin{lemma} \label{lem:deltanat}
    If $Q$ is a flat involutive MH complex, then $\Delta$ is a natural transformation from $1_{\c C(Q)}$ to $\phi$.
\end{lemma}

\begin{proof}
    First, it is clear from the definition that $\Delta(x)$ is a morphism from $x$ to $\phi(x)$. So it remains to verify that for $x,y \in \Obj(\c C(Q))$ and $x \xrightarrow{f} y$, we have
    \[
        f \Delta(y) = \Delta(x) \phi(f).
    \]
    After translating notation, this is the statement of \cite[Prop.~25(i)]{salvetti1993homotopy}.
\end{proof}

The last, and least trivial, portion of the argument is to show that $(\c C_{x \rightarrow}, \preccurlyeq)$ and $(\c C_{\rightarrow y}, \succcurlyeq)$ are lattices. To do this, we adapt arguments of Deligne \cite{deligne1972immeubles}. By Proposition \ref{prop:posetsiso}, it suffices to show that $(\c C_{x \rightarrow}, \preccurlyeq)$ is a lattice.

\textbf{In the rest of this section}, let $Q$ be FISMH and $\c C = \c C(Q)$.

\begin{prop}\label{prop:startingminimal}
    \emph{\cite[Thm.~31(iii)]{salvetti1993homotopy}}
    Let $f \in \c C_{x \rightarrow}$. Then there is a $y \in \Obj(\c C)$ satisfying the following property: $u(x,z) \preccurlyeq f$  
    if and only if $u(x,z) \preccurlyeq u(x,y)$.
\end{prop}

The following is the main result which we use when showing $\c C_{x \rightarrow}$ is a lattice.

\begin{prop}  \label{prop:preglb}
    \emph{\cite[Prop 1.14, Cor 1.20]{deligne1972immeubles}}
    Let $x \in \Obj(\c C)$ and $S \subseteq \c C_{x \rightarrow}$ a set of morphisms starting at $x$ with bounded length (meaning there is some $N$ so that $\ell(f) \leq N$ for all $f \in S$). Then $S$ satisfies
    \begin{enumerate}
        \item $\mathrm{id}_x \in S$,
        \item If $g \in S$ and $f \preccurlyeq g$, then $f \in S$; and
        \item Suppose $x \xrightarrow{g} y$ and $y_1,y_2$ are vertices adjacent to $y$ in $Q$. Let $e \in \c Q$ be the (unique) 2-cell containing $y$, $y_1$, and $y_2$ (this exists since $Q$ is simplicial). Let $z = \tbar w(y,e)$. If $g u(y,y_i) \in S$ for $i=1,2$, then $g u(y, z) \in S$.
    \end{enumerate}
    if and only if there is a (unique) morphism $h \in \c C_{x \rightarrow}$ such that
    \begin{align*}
        S = \{\,g \in \c C_{x \rightarrow} : g \preccurlyeq h\,\}
    \end{align*}
\end{prop}

For the forward implication, we mimic the proof of Proposition 1.14 of \cite{deligne1972immeubles}; for the reverse, we mimic the proof of Corollary 1.20 of \cite{deligne1972immeubles}.

\begin{proof}
    Suppose (1), (2), and (3) are satisfied. Let $h \in S$ with maximal length (this exists since $S$ has bounded length). We claim this $h$ works. It suffices to show that
    \begin{enumerate}
        \item[($*$)] Suppose $x \xrightarrow{g} y$ and $y_1 \in \Obj(\c C)$ is adjacent to $y$ such that (a) $g \preccurlyeq h$ but (b) $g u(y,y_1) \not\preccurlyeq h$ and $g u(y,y_1) \in S$. Then there exists $g'$, $y'$, and $y'_1$ satisfying (a) and (b) with $g'$ strictly longer than $g$.
    \end{enumerate}
    Since $g u(y,y_1) \in S$, the maximality of $\ell(h)$ implies $g \not= h$. Thus there is some $y_2 \not= y_1$ adjacent to $y$ such that $g u(y,y_2) \preccurlyeq h$. Letting $e$ be the 2-cell of $Q$ containing the vertices $y,y_1,y_2$ and $z = \tbar w(y,e)$, this implies that $g u(y,z) \in S$ by (3).
    Note that since $g u(y,y_1) \preccurlyeq g u(y,z)$, we must have that $g u(y,z) \not\preccurlyeq h$.

    Consider $u(y,z)$. Let $y_0$ be a vertex of $e$ so that $u(y,y_2) \preccurlyeq u(y,y_0)$, and $u(y,y_0)$ has the maximal length among paths $p \preccurlyeq u(y,z)$ satisfying $gp \preccurlyeq h$. Let $g' = g u(y,y_0)$. Let $y'$ be the vertex of $u(y',z)$ adjacent to $y'$. Then $g'$ and $y'$ satisfy (a) and (b) above. Uniqueness of $h$ is immediate.

    Now suppose $S = \{\,g \in \c C_{x \rightarrow} : g \preccurlyeq h\,\}$ for some $h$. (1) and (2) are clear, so we show that $S$ satisfies (3). Suppose $g,y,y_i,z$ are as given in (3). We must show that $g u(y,z) \in S$. 

    Suppose $g \in S$, say $h = gf$ for some morphism $f$ starting at $y$. Then by cancellation we must have that $u(y,y_i) \preccurlyeq f$ for $i=1,2$.
    Let $y_0$ be the vertex guaranteed to exist by Proposition \ref{prop:startingminimal} for $f$, so that in particular $u(y,y_0) \preccurlyeq f$.
    One may easily verify that because $y_1$ and $y_2$ are distinct vertices of $e$ and $u(y,y_i) \preccurlyeq f$, we must have that $y_0 = \tbar w(y,e) = z$. 
\end{proof}

From this, the following is immediate.

\begin{cor} \label{cor:glb}
    Every pair in $\c C_{x \rightarrow}$ has a unique greatest lower bound under $\preccurlyeq$.
\end{cor}

\begin{proof}
Let $f,g \in \c C_{x \rightarrow}$ and let
\begin{align*}
    S_f &= \{\,h \in \c C_{x \rightarrow} : h \preccurlyeq f \,\} \\
    S_g &= \{\,h \in \c C_{x \rightarrow} : h \preccurlyeq g \,\} \\
    S &= S_f \cap S_g.
\end{align*}
It is clear that $S$ has bounded length and satisfies (1), (2), and (3) of Proposition \ref{prop:preglb} since both $S_f$ and $S_g$ do. Thus by Proposition \ref{prop:preglb} there is a unique element, which we denote $f \wedge g$ such that
\begin{align*}
    S = \{\,h \in \c C_{x \rightarrow} : h \preccurlyeq f \wedge g \,\}.
\end{align*}
This $f \wedge g$ is clearly a unique greatest lower bound for the pair $f,g$ under $\preccurlyeq$.
\end{proof}

In order to show that $\c C_{x \rightarrow}$ contains joins, we need the following

\begin{prop} \label{prop:hasupperbound}
    \emph{\cite[Prop.~25(ii)]{salvetti1993homotopy}} If $Q$ is a flat involutive MH complex and $c \in \c C_{x \rightarrow y}$, then if $f \in \c C_{x \rightarrow}$ such that
    \[
        f = u(x, x_1) u(x_1,x_2) \dots u(x_{n-1},x_n)
    \]
    for some objects $x_i$, we have that $f \preccurlyeq c\Delta^n$ (where $\Delta = \Delta_y$ or $\Delta^y$ depending on its place in the product).
\end{prop}

\begin{cor} \label{cor:lub}
    Every pair in $\c C_{x \rightarrow}$ has a unique least upper bound under $\preccurlyeq$.
\end{cor}

\begin{proof}
Let $f,g \in \c C_{x \rightarrow}$. Proposition \ref{prop:hasupperbound} shows that there exists a common upper bound of $f$ and $g$: if $n_f,n_g$ are the integers satisfying $f \preccurlyeq \Delta^{n_f}$ and $g \preccurlyeq \Delta^{n_g}$, and if $n = \max\{n_f,n_g\}$ then both $f \preccurlyeq \Delta^n$ and $g \preccurlyeq \Delta^n$.

Let $\c S$ be the collection of all sets $S$ of bounded length with $f,g \in S$ and satisfying (1), (2), and (3) of Proposition \ref{prop:preglb}. Note that $\c S$ is non-empty, since 
\[
    \{\,h \in \c C_{x \rightarrow} : h \preccurlyeq \Delta^n \,\}
\]
is contained in $\c S$.
 Let $S_0 = \bigcap_{S \in \c S} S$. Then $S_0$ has bounded length, satisfies (1),(2),(3), and contains $f,g$ since each $S \in \c S$ satisfies each of these.
By Proposition \ref{prop:preglb}, there is a unique morphism, which we call $f \vee g$, such that
\begin{align*}
    S_0 = \{\, h \in \c C_{x \rightarrow} : h \preccurlyeq f \vee g \,\}.
\end{align*}
Certainly $f \preccurlyeq f \vee g$ and $g \preccurlyeq f \vee g$, so $f \vee g$ is an upper bound of $f$ and $g$. If $u$ is another upper bound of $f$ and $g$, then
\begin{align*}
    S' \coloneqq \{\, h \in \c C_{x \rightarrow} : h \preccurlyeq u \,\} \in \c S,
\end{align*}
implying $S_0 \subseteq S'$, and in particular $f \vee g \in S'$, hence $f \vee g \preccurlyeq u$. 
\end{proof}

We can now conclude by proving Theorem \ref{thm:fismhgarside}.

\begin{proof}[Proof (of Theorem \ref{thm:fismhgarside})]
    We have seen that $\phi$ is an automorphism of $\c C(Q)$. Lemma \ref{lem:deltanat} shows that $\Delta$ is a natural transformation from $1_{\c C(Q)}$ to $\phi$. Lemma \ref{lem:mhhom} shows that $\c C(Q)$ is homogeneous and Lemma \ref{lem:qcancel} shows that $\c C(Q)$ is cancellative. Lemma \ref{lem:atomsaresimple} implies that all atoms of $\c C(Q)$ are simple. Corollaries \ref{cor:glb} and \ref{cor:lub} imply that $(\c C(Q)_{x \rightarrow}, \preccurlyeq)$ is a lattice, and their dual statements (and proofs) imply that $(\c C(Q)_{\rightarrow y}, \succcurlyeq)$ is a lattice. Therefore $\c C(Q)$ is a Garside category.
\end{proof}

The following is the main consequence that we're interested in.

\begin{cor} \label{thm:salcpxgarside}
    If $Q$ is FISMH, then for any vertex $x$, $\pi_1(\sal(Q),x) \cong \c G_{x}$. In particular, $\pi_1(\sal(Q))$ is a weak Garside group.
\end{cor}

This follows from

\begin{lemma}
    \emph{\cite[Cor.~18]{salvetti1993homotopy}} If $Q$ is an MH complex, then all relations for $\pi_1(\sal(Q))$ come from the set $\{\gamma_1 \gamma_2^{-1} : \gamma_1, \gamma_2 \text{ are positive paths with the same endpoints}\,\}$.
\end{lemma}

In other words, homotopies between loops in $\sal(Q)$ are equivalent to $\sim$ equivalences of positive subpaths. Since $\c G_x$ consists of $\sim$ equivalence classes of positive paths and the formal inverses of these equivalence classes, Corollary \ref{thm:salcpxgarside} follows.

\section{Application to arrangements and matroids}
\label{sec:matroids}

We now discuss applications to simplicial pseudohyperplane arrangements and their associated matroids. First, we will give further terminology regarding these arrangements.

Let $\c A$ be a pseudosphere arrangement. We say that $\c A$ is \emph{proper} (also called \emph{essential} in some texts) if $\bigcap_{S \in \c A} S = \varnothing$. We say that $\c A$ is centrally symmetric if each element $S \in \c A$ is preserved under the central symmetry $-1$ of $\mathbb{S}^{d}$, that is, if $-S = S$ for all $S \in \c A$. We call two arrangements $\c A_1, \c A_2$ \emph{homeomorphic}, and write $\c A_1 \cong \c A_2$, if there is a homeomorphism of $\mathbb{S}^{d}$ which is a bijection between $\c A_1$ and $\c A_2$.

An \emph{(open) chamber} of $\c A$ is a connected component of $\mathbb{S}^{d} \setminus \bigcup_{S \in \c A} S$. 
The union of the closures of the chambers gives a cellulation of $\mathbb{S}^{d}$ (where the pseudospheres of $\c A$ are codimension-1 subcomplexes).
We call $\c A$ \emph{simplicial} if this induced cell structure on $\mathbb{S}^{d}$ is.
The \emph{dual complex} $Q(\c A)$ to $\c A$ is the dual of this cell structure on $\mathbb{S}^{d}$.

We say that the pseudohyperplane arrangement $c \c A$ satisfies the above properties (e.g., proper, simplicial, etc.) if $\c A$ does. A chamber of $c \c A$ is the cone on a chamber of $\c A$. We define $Q(c \c A) = Q(\c A)$.

We note that centrally symmetric simplicial arrangements are always proper, since the intersection of centrally symmetric pseudospheres which bound a simplex must have empty intersection.
We also note that two arrangements are homeomorphic if and only if their dual complexes are, hence homeomorphism is determined combinatorially. 

The following connects arrangements and MH complexes.

\begin{prop}
    \emph{\cite[Prop.~6]{salvetti1993homotopy}}
    Suppose $\c A$ is an arrangement of pseudospheres in $\mathbb{S}^{d}$. Then the complex $Q(\c A)$ dual to $\c A$ is MH.
\end{prop}

Thus, as with the other MH complexes we have considered, there is a Salvetti complex $\sal(Q(\c A))$, which for brevity we denote $\sal(\c A)$. (We define $\sal(c\c A) = \sal(\c A)$.) We now determine those arrangements $\c A$ which give FISMH complexes $Q(\c A)$, and hence for which Corollary \ref{thm:fismhgarside} holds.

\begin{lemma}
    \emph{\cite[Thm.~20]{salvetti1993homotopy}} Suppose $\c A$ is a proper pseudosphere arrangement. Then $Q(\c A)$ is flat.
\end{lemma}

The following is clear from the definitions.

\begin{lemma}
    If $\c A$ is a proper centrally symmetric pseudosphere arrangement, then $Q(\c A)$ is involutive with involution $\phi(x) = -x$.
\end{lemma}

Since centrally symmetric simplicial arrangements are always proper,
we have the following
 
\begin{prop} \label{thm:simppseudogarside}
    If $\c A$ is a centrally symmetric simplicial pseudosphere arrangement, then $\pi_1(\sal(\c A)) = \pi_1(\sal(c\c A))$ is a weak Garside group.
\end{prop}

We can connect the topology of $\sal(c \c A)$ more directly with the arrangement $c \c A$ itself. This connection is completely analagous to the connection between a hyperplane arrangement and its Salvetti complex via the complement of its complexification. 

\begin{defn}
    Let $c \c A$ be a pseudohyperplane arrangement in $\mathbb{R}^d$. For $H \in c \c A$, define
    \begin{align*}
        H_{\mathbb{C}} = H + iH \subseteq \mathbb{C}^d
    \end{align*}
    (where $H + iH = \{\,h_1 + ih_2 : h_i \in H\,\}$).
    Then, we define
    \begin{align*}
        M(c\c A) = \mathbb{C}^d \setminus \bigcup_{H \in c \c A} H_{\mathbb{C}}.
    \end{align*}
\end{defn}

We would like to say that $M(c \c A)$ is homotopy equivalent to $\sal(c \c A)$, but this is not quite true. Instead, we need the following variation.

Let $\hat Q = \hat Q(\c A) = \hat Q(c \c A)$ be the cellulation of the unit ball $D^{d+1}$ of $\mathbb{R}^{d+1}$ obtained by identifying $\partial D^{d+1}$ with the cellulation $Q(\c A)$ of the unit sphere $\mathbb{S}^{d}$. We have the following

\begin{prop}
    \emph{\cite[Prop.~9]{salvetti1993homotopy}}
    The complex $\hat Q(\c A)$ is an MH complex with exactly one cell (namely, $D^{d+1}$) of top dimension.
\end{prop}

Thus we can define $\hsal(c \c A) = \hsal(\c A) \coloneqq \sal(\hat Q(\c A))$, which we call the \emph{completed Salvetti complex}. Note that the $d$-skeletons of $Q$ and $\hat Q$ are the same, so the $d$-skeletons of $\sal(\c A)$ and $\hsal(\c A)$ are the same. In particular, since we're assuming $d \geq 2$, the 2-skeletons are the same, so $\pi_1(\hsal(\c A)) \cong \pi_1(\sal(\c A))$.
The relevant result for us is the following (translated into our notation)

\begin{prop}
    \emph{\cite[Thm.~2.30]{deshpande2016arrangements}} If $c \c A$ is a pseudohyperplane arrangement, then $M(c \c A)$ is homotopy equivalent to $\hsal(c \c A)$.
\end{prop}

In particular, $\pi_1(M(c \c A)) \cong \pi_1(\sal(c \c A))$. Moreover, \cite[\S 6]{salvetti1993homotopy} implies $\hsal(c \c A)$ is a $K(\pi,1)$ for simplicial arrangements, i.e., has contractible universal cover. (This was also shown in \cite{deshpande2016arrangements}.) So, we can summarize all of the above results (in addition to the known result for $d = 1$) into the following

\begin{thm} \label{thm:mainconsequence1}
    If $c \c A$ is a proper simplicial pseudohyperplane arrangement, then $\pi = \pi_1(M(c \c A))$ is a weak Garside group and $M(c \c A)$ is a $K(\pi,1)$.
\end{thm}

The combinatorial information contained in a pseudosphere arrangement can be extracted to objects known as \emph{oriented matroids}, which have broad uses in mathematics. We recall this connection here, and explain how our above results apply to this setting.

\begin{defn}
    An oriented matroid is a triple $\c M = (E, \c E, {}^*)$ where $E$ is a finite set, $\c E$ a collection of non-empty subsets of $E$, and ${}^* : E \to E$ a ``fixed point-free'' involution of $E$ (so $(x^*)^* = x$ and $x^* \not= x$) such that
    \begin{enumerate}
        \item If $A,B \in \c E$ and $A \subseteq B$, then $A = B$,
        \item If $S \in \c E$ then $S^* \in \c E$ and $S \cap S^* = \varnothing$, and
        \item For $S,T \in \c E$ with $x \in S \cap T^*$ and $S \not= T^*$, there is a $C \in \c E$ with $C \subseteq (S \cup T) \setminus \{x,x^*\}$.
    \end{enumerate}
    We call $E$ the \emph{ground set} and the elements of $\c E$ the \emph{circuits} of the oriented matroid $\c M$. The \emph{rank} $\rk(\c M)$ of $\c M$ is the minimal cardinality of a set in $\c E$. An \emph{isomorphism} between oriented matroids is a bijection between ground sets which induces a bijection on the circuits and commutes with the involution.
\end{defn}

Let $\c A$ be a pseudosphere arrangement in $\mathbb{S}^{d}$ with each element endowed with a choice of orientation. Let
\begin{align*}
    E = E(\c A) = \{\,H : H \text{ a halfspace of } S \in \c A \,\}.
\end{align*}
Define ${}^*$ on $E$ by exchanging the halfspaces of any given pseudosphere, e.g., if $H_1,H_2$ are the halfspaces of $S$ then $H_1^* = H_2$ and $H_2^* = H_1$.
We then let $\c E = \c E(\c A)$ be the collection of minimal subsets $C$ of $E$ with
\begin{enumerate}
     \item $C \cap C^* = \varnothing$, and
     \item $\bigcup_{H \in C} \overline H = \mathbb{S}^{d}$,
 \end{enumerate} 
where $\overline H$ is the topological closure of the halfspace $H$ in $\mathbb{S}^{d}$. 
By \cite[Thm.~16]{folkman1978oriented}, $\c M(\c A) \coloneqq (E(\c A), \c E(\c A), {}^*)$ is an oriented matroid of rank $d+1$. ($\c M(\c A)$ is not to be confused with $M(c \c A)$.) The content of the Folkman-Lawrence realization theorem is the following.

\begin{prop} \emph{\cite[Thm.~20]{folkman1978oriented}} \label{prop:toprep}
    Let $\c M = (E, \c E, {}^*)$ be an oriented matroid with 
    $\mathrm{rk}(\c M) = r$. Then there is a proper central pseudosphere arrangement $\c A(\c M)$ in $\mathbb{S}^{r-1}$ such that the set of cells of the corresponding dual complex $Q(\c A(\c M))$ are order-isomorphic to the subsets $C$ of $\c E$ with $C \cap C^* = \varnothing$, with vertices of $Q(\c A(\c M))$ mapping to elements of $\c E$. 
    Moreover, if $\c M$ is an oriented matroid, then
    \[
        \c M(\c A(\c M)) \cong \c M,
    \]
    and if $\c A$ is a proper pseudosphere arrangement, then
    \[
        \c A(\c M(\c A)) \cong \c A
    \]
    (where $\cong$ denotes isomorphism and homeomorphism, respectively). 
\end{prop}

In particular, two pseudosphere arrangements are homeomorphic if and only if their corresponding oriented matroids are isomorphic. 

This allows us to sensibly define $\sal(\c M) = \sal(\c A(\c M))$ and $\hsal(\c M) = \hsal(\c A(\c M))$ for any oriented matroid $\c M$. By unraveling definitions, one sees that these are equivalent to typical definitions of a Salvetti complex of an oriented matroid \cite[\S 5.5]{ziegler1987algebraic}. Thus Theorem \ref{thm:mainconsequence1} can be rephrased in terms of matroids as follows.

\begin{thm} \label{thm:simpmatgar}
    If $\c M$ is an oriented matroid and its pseudosphere realization $\c A(\c M)$ is simplicial, then $\pi = \pi_1(\sal(\c M))$ is a weak Garside group and $\hsal(\c M)$ is a $K(\pi,1)$.
\end{thm}

\bibliographystyle{alpha}
\bibliography{PHs}

\end{document}